\documentclass[reqno, 12pt]{article} 

\pdfoutput=1

\usepackage{enumerate}
\usepackage{latexsym}
\usepackage[centertags]{amsmath}
\usepackage{amsfonts}
\usepackage{amssymb}
\usepackage{amsthm}
\usepackage{newlfont}
\usepackage{graphics}
\usepackage{color}
\usepackage{float}
\usepackage{diagbox}
\usepackage{seqsplit}
\usepackage{hyperref}
\usepackage{longtable}
\usepackage{rotating}
\usepackage{multirow}
\usepackage{extarrows}
\usepackage[sort,compress,numbers]{natbib}
\usepackage[utf8]{inputenc}

\newtheorem{thm}{Theorem}[section]

\newtheorem{lem}[thm]{Lemma}
\newtheorem{prop}[thm]{Proposition}

\theoremstyle{mydefinition}

\theoremstyle{myremark}

\allowdisplaybreaks[4]

\def\CT{\mathop{\mathrm{CT}}}

\title{Meeting a Challenge raised by Ekhad and Zeilberger related to Stern's Triangle}

\author{Jinlong Tang$^{1}$ and Guoce Xin$^{2,}$
\\[2mm]
{\small $^{1, 2}$ School of Mathematical Sciences,}\\[-0.8ex]
{\small Capital Normal University, Beijing, 100048, P.R.~China}
}

\date{June 10, 2025}

\begin{document}

\maketitle

\begin{abstract}
This paper resolves an open problem raised by Ekhad and Zeilberger for computing $\omega(10000)$, which is related to Stern's triangle. While $\nu(n)$, defined as the sum of squared coefficients in $\prod_{i=0}^{n-1} (1 + x^{2^i} + x^{2^{i+1}})$, admits a rational generating function, the analogous function $\omega(n)$ for $\prod_{i=0}^{n-1} (1 + x^{2^i+1} + x^{2^{i+1}+1})$ presents substantial computational difficulties due to its complex structure.

We develop a method integrating constant term techniques, conditional transfer matrices, algebraic generating functions, and $P$-recursions.
Using the conditional transfer matrix method, we represent $\omega(n)$ as the constant term of a bivariate rational function. This framework enables the calculation of $\omega(10000)$, a $6591$-digit number, and illustrates the method's broad applicability to combinatorial generating functions.
\end{abstract}

\noindent
\begin{small}
 \emph{Mathematic subject classification}: Primary 05A15; Secondary 05A10.
\end{small}

\noindent
\begin{small}
\end{small}

\section{Introduction}
In \cite{RP.Stanley}, Stanley defined \emph{Stern's triangle} as an array of numbers analogous to Pascal's triangle. The generating function for the $n$th row (starting with $n=0$) is given by
\[
F_n(x) := \sum_{k \geq 0} a(n,k) x^k = \prod_{i=0}^{n-1} (1 + x^{2^i} + x^{2^{i+1}}).
\]
This array is named \emph{Stern's triangle} because, in the limit as $n \to \infty$,
\[
F(x) = \sum_{k \geq 0} a_k x^k = \prod_{i=0}^{\infty} (1 + x^{2^i} + x^{2^{i+1}})
\]
yields the generating function for the celebrated Stern diatomic sequence. This sequence exhibits the remarkable property (not irrelevant here)  that every positive rational number appears exactly once as a ratio $a_k / a_{k+1}$, and $\gcd(a_k, a_{k+1}) = 1$.

Stanley initiated the study of $\nu(n) = \sum_k a(n,k)^2$ and established that $\nu(n)$ possesses a rational generating function:
\[
\sum_{n \geq 0} \nu(n) x^n = \frac{1 - 2x}{1 - 5x + 2x^2}.
\]
His proof relies on the recursion $F_n(x) = (1 + x + x^2) F_{n-1}(x^2)$, and this approach extends to sums of the form
\begin{align}
u_\alpha(n) := \sum_k a(n,k)^{\alpha_0} a(n,k+1)^{\alpha_1} \cdots a(n,k+m-1)^{\alpha_{m-1}},\label{eq:u-alpha}
\end{align}
where $\alpha \in \mathbb{N}^{m}$. Note that when $\alpha = (2)$, $u_{\alpha}(n)$ reduces to $\nu(n)$. This method also generalizes to other Stern-like sequences.

Subsequently, Ekhad and Zeilberger \cite{stern} provided an alternative proof using the recursion
\[
F_n(x) = F_{n-1}(x) (1 + x^{2^{n-1}} + x^{2^n}).
\]
They further developed a Maple procedure to automate the computation of $u_\alpha(n)$. In the same paper, they proposed the following prized problem:

\medskip\noindent
\textbf{The Prized Problem:}
Compute $\omega(10000)$, where
\[
G_n(x) = \sum_{k \geq 0} b(n,k) x^k := \prod_{i=0}^{n-1} (1 + x^{2^i + 1} + x^{2^{i+1} + 1}), \quad \omega(n) := \sum_{k \geq 0} b(n,k)^2.
\]

\medskip\noindent
This problem illustrates how a slight modification (a proverbial ``$\varepsilon$") can transform a tractable problem into an intractable one. Although $G_n(x)$ resembles $F_n(x)$, computing $\omega(10000)$ proves substantially harder than $\nu(10000)$. The unconventional form of $G_n(x)$ obscures any recognizable patterns, rendering standard generating function techniques ineffective.

Using the definition, Maple computes the first $21$ values of $\omega(n)$ (starting at $n=0$):
\[
1,  3,  13,  55,  249,  1121,  5025,  22607,  101931,  460877,  2088687,  9482763,  43109307,  196163983,
\]
\[
893222041,  4069162197,  18543631161,  84525140297,  385343891847,  1756959373157,  8011450183181.
\]
Computing further values becomes prohibitive due to exponential growth in both polynomial degree and coefficient size.

In this paper, we solve this problem by computing $\omega(10000)$ via constant term manipulation. Our approach also uses the conditional transfer matrix method, algebraic generating functions, and P-recursions.

The remainder of the paper is organized as follows. Section 2 presents a constant term approach to $\nu(n)$ starting with the observation $\nu(n) = \CT_x F_n(x) F_n(x^{-1})$. The proof is similar to that of Ekhad and Zeilberger. We then analyze the obstacles in adapting this approach to $\omega(n)$. Section 3 expresses $\omega(n)$ as the constant term of a three-variable iterated Laurent series with auxiliary variables $u$ and $v$ (later substituted by $x$). This enables the application of the transfer matrix method under specific constraints, yielding $\omega(10000)$ as the constant term of a complex rational function. Section 4 evaluates this constant term using algebraic generating functions and P-recursions.

\def\ldeg{\textrm{ldeg}}
\section{Constant term approach to $\nu(n)$}\label{2}
In this section, we recompute $\nu(n)$ using the constant term operator $\CT_x$ and
explain the difficulty for computing $w(n)$.

A Laurent polynomial in $x$ is of the form $L(x)=\sum_{i=M}^{N} a_i x^i$ where
the coefficients $a_i$ belong to a ring or a field. The degree
$\deg L(x)$ is $N$ if $a_N\neq 0$ and the least degree $\ldeg L(x)$ is $M$ if $a_M\neq 0$. The constant term (linear) operator $\CT_x$ acts by
$$ \CT_x \sum_{i=M}^{N} a_i x^i =a_0.$$

The operator $\CT_x$ naturally extend for Laurent series (allowing $N\to \infty$), or even for formal Laurent series (allowing $M\to -\infty$ and $N\to \infty$).

Our approach to $\nu(n)$ is based on the
observation
$$\sum_{i=M}^{N} a_i^2 = \CT_x L(x) L(x^{-1}),$$
and the following two basic properties
\begin{enumerate}
  \item[P1] $\CT_x L(x) = \CT_{x} L(x^{-1})$ for any Laurent polynomial $L(x)$.
  \item[P2] $\CT_x L(x) = 0$ whenever $\deg L(x) < 0$ or $\ldeg L(x) > 0$.
\end{enumerate}

We derive the recurrence relation for $\nu(n)$ through the following equalities:
\begin{align}
   \nu(n+1) & = \CT_{x} F_{n+1}(x)F_{n+1}\left(\frac{1}{x}\right) \notag \\
   & = \CT_{x} F_{n}(x)F_{n}\left(\frac{1}{x}\right)\left(x^{-2^{n+1}} + 2x^{-2^{n}} + 3 + 2x^{2^{n}} + x^{2^{n+1}}\right) \notag\\
   & = \CT_{x} F_{n}(x)F_{n}\left(\frac{1}{x}\right)\left(3 + 4x^{2^{n}} + 2x^{2^{n+1}}\right) \notag\\
   & = \CT_{x} F_{n}(x)F_{n}\left(\frac{1}{x}\right)\left(3 + 4x^{2^{n}}\right).\label{res}
\end{align}
The elimination of the $x^{2^{n+1}}$ term in the last step follows from P2, since
$\deg F_n(x)=2^{n+1}-2$ implies that $\ldeg \left[ F_{n}(x)F_{n}(x^{-1})x^{2^{n+1}} \right] = 2 > 0$.

Define the auxiliary quantity $\nu^{(1)}(n) := \CT_{x} F_{n}(x)F_{n}\left(\frac{1}{x}\right)x^{2^{n}}$. From equation (\ref{res}), we directly obtain the primary recurrence:
\[
\nu(n+1) = 3\nu(n) + 4\nu^{(1)}(n).
\]
An analogous argument applied to $\nu^{(1)}(n+1)$ yields the secondary recurrence:
\[
\nu^{(1)}(n+1) = \nu(n) + 2\nu^{(1)}(n).
\]

Let $\Gamma(n) := \begin{pmatrix} \nu(n) \\ \nu^{(1)}(n) \end{pmatrix}$ and define the transition matrix $M := \begin{pmatrix} 3 & 4 \\ 1 & 2 \end{pmatrix}$. We obtain the matrix recurrence relation
\[
\Gamma(n+1) = M\Gamma(n).
\]
The initial conditions $\nu(1) = 3$ and $\nu^{(1)}(1) = 1$ extend naturally to $n=0$ by defining $\nu(0) = 1$ and $\nu^{(1)}(0) = 0$. Consequently, the closed-form expression becomes
\[
\Gamma(n) = M^n\Gamma(0).
\]

To compute $M^n$, we consider the matrix generating series:
\begin{align}
\sum_{n \geq 0} M^n t^n &= (I - tM)^{-1} \notag \\
&= \det(I - tM)^{-1} (I - tM)^*,
\end{align}
where $(I - tM)^*$ denotes the adjoint matrix, whose entries are polynomials in $t$.

Through this framework, we derive the generating function for $\Gamma(n)$:
\[
\sum_{n \geq 0} \Gamma(n)x^n = (I - Mx)^{-1}\Gamma(0) = \begin{pmatrix} \dfrac{1 - 2x}{1 - 5x + 2x^2} \\ \dfrac{x}{1 - 5x + 2x^2} \end{pmatrix}.
\]
Extracting the first component yields the generating function for our main sequence:
\[
\sum_{n \geq 0} \nu(n)x^n = \frac{1 - 2x}{1 - 5x + 2x^2}.
\]
This generating function, obtained through the constant term method, coincides with the result established by Stanley.

In the computation of the state $\nu(n)$, we introduce an auxiliary state $\nu^{(1)}(n)$ and establish a system of recursive relations involving these two states. By employing the method of transformation matrices, we can systematically solve for $\nu(n)$ in closed form. This approach is versatile and can be extended to a broader class of problems, provided that the number of states involved remains finite.

\medskip
The computation becomes more complex when we replace $F_n(x)$ with $G_n(x)$. The degree property now reads
\[
\deg G_n(x) = 2^{n+1} - 2 + n.
\]
Attempting to proceed as before leads to infinitely many states of the form $\CT_x G_n(x)G_n(x^{-1}) x^{2^n+k}$ for integers $k$ depending on $n$. This prevents the previous approach from computing $\omega(10000)$ effectively.

Our key innovation is the introduction of an auxiliary variable $u$, which will ultimately be substituted by $x$. We demonstrate that under certain conditions, $u$ can be treated as a constant, enabling the application of the transformation matrix method. Similarly, we express $\omega(10000)$ as the constant term of a rational function. However, computing $\omega(10000)$ remains challenging. We therefore divide the computation across the next several sections.

\section{Conditional Transformation Matrix Method}
In this section, we establish Theorem \ref{Recurrence-Matrix} using the conditional transformation matrix method. This allows us to express $w(n)$ as the constant term of a rational function.

Define the auxiliary function
\[
K_{n}(u,x):=\prod_{i=0}^{n-1}\left(1 + u x^{2^{i}} + u x^{2^{i+1}}\right).
\]
Observe that $G_n(x) = K_n(x,x)$.

Let us derive an alternative expression of $\omega(n)$. This is better work in
the field of iterated Laurent series $\mathbb{R}((x_1))((x_2))((x_3))$, where $x_1=u,x_2=x, x_3=v$ in our case. This field includes the field of rational functions as a subfield, thus any rational function has a unique
formal Laurent series expansion. Consequently, the constant term operators
$\CT_{x_i}$ commutes with each other. See \cite{Xin-CT} for details.

Here we only need to keep in mind the following two series expansions and Lemma \ref{P-x,u}:
$$\frac{1}{1-x/u}= \sum_{i\geq0}(x/u)^{i}, \text{ and }  \frac{1}{1-v/x} = \sum_{i\geq0}(v/x)^{i}. $$

\begin{lem}\label{P-x,u}
If $P(u)=\sum_{i\geq0}a_{i}u^{i}$, where $a_{i}$ is free of $u$ for all $i\geq0$, then
$$\CT\limits_{u}P(u)\cdot\frac{1}{1-x/u}=P(x).$$
Similarly, if $N(u)=\sum_{i\leq0}b_{i}v^{i}$, where $b_{i}$ is free of $v$ for all $i\leq0$, then
$$\CT\limits_{v}N(v)\cdot\frac{1}{1-v/x}=N(x).$$
\end{lem}

\begin{proof}
The first equation follows by linearity and the following simple fact:
  \begin{align*}
  \CT\limits_{u}u^{k}\cdot\frac{1}{1-x/u}=\CT\limits_{u}u^{k}\cdot\sum_{i\geq0}(x/u)^{i}=
  \begin{cases}
    x^{k} & \text{if } k\geq0, \\
    0 & \text{if } k<0.
  \end{cases}
  \end{align*}
The second equation follows in a similar way.
\end{proof}

\begin{lem}
$$\omega(n)=\CT_{u,v,x} K_{n}(u,x) K_{n}(v^{-1},x^{-1}) \frac{1}{1-\frac{x}{u}} \frac{1}{1-\frac{v}{x}}$$
\end{lem}
\begin{proof}
We have
\begin{align*}
\CT_{u,v,x} K_{n}(u,x) K_{n}(v^{-1},x^{-1}) \frac{1}{1-\frac{x}{u}} \frac{1}{1-\frac{v}{x}}
&= \CT_{x} \left( \CT_{u} K_{n}(u,x) \frac{1}{1-\frac{x}{u}} \right) \left( \CT_{v} K_{n}(v^{-1},x^{-1}) \frac{1}{1-\frac{v}{x}} \right) \\
&= \CT_{x} K_{n}(x,x) K_{n}(x^{-1},x^{-1}) \\
&= \CT_{x} G_{n}(x) G_{n}(x^{-1}) = \omega(n).
\end{align*}
\end{proof}

We introduce a state vector $\Phi_{n} := \left( \varphi_{n}^{(-2)}, \varphi_{n}^{(-1)}, \dots, \varphi_{n}^{(2)} \right)^{\mathsf{t}}$, where each entry is defined as
\[
\varphi_{n}^{(i)} := x^{i \cdot 2^{n}} K_{n}(u,x) K_{n}(v^{-1},x^{-1}) \frac{1}{1-\frac{x}{u}} \frac{1}{1-\frac{v}{x}}.
\]
These states satisfy a recursive relationship expressed in the following proposition.

\begin{prop}\label{Recurrence-Matrix-uv}
Following the notation above, for all $m, n \in \mathbb{Z}_{+}$ with $m < 2^{n} + 2 - n$, we have
\[
\CT_{u,v,x} \Phi_{m+n} = \CT_{u,v,x} A^{m} \Phi_{n},
\]
where the transformation matrix $A$ is given by
\[
A = \begin{pmatrix}
u & 0 & 0 & 0 & 0 \\
1 + 2\frac{u}{v} & u + \frac{u}{v} & u & 0 & 0 \\
v^{-1} & v^{-1} + \frac{u}{v} & 1 + 2\frac{u}{v} & u + \frac{u}{v} & u \\
0 & 0 & v^{-1} & v^{-1} + \frac{u}{v} & 1 + 2\frac{u}{v} \\
0 & 0 & 0 & 0 & v^{-1}
\end{pmatrix}.
\]
\end{prop}

\begin{proof}
First, note that entries of $A$ are linear in $u$ and $v^{-1}$. Consequently, entries of $A^{m-1}$ are polynomials in $u$ (respectively $v^{-1}$) of degree at most $m-1$. For each $i$ from 1 to 5, consider the $i$-th row of $A^{m-1}\Phi_{n+1}$:
\begin{align}
(A^{m-1}\Phi_{n+1})_{i} = (A^{m-1})_{(i,1)}\varphi_{n+1}^{(-2)} + (A^{m-1})_{(i,2)}\varphi_{n+1}^{(-1)} + \cdots + (A^{m-1})_{(i,5)}\varphi_{n+1}^{(2)}.\label{eq:A-i-row}
\end{align}
Focusing on the first term:
\begin{align}
(A^{m-1})_{(i,1)}\varphi_{n+1}^{(-2)}
&= (A^{m-1})_{(i,1)} \left[ x^{-2\cdot2^{n+1}} K_{n+1}(u,x) K_{n+1}(v^{-1},x^{-1}) \frac{1}{1 - \frac{x}{u}} \frac{1}{1 - \frac{v}{x}} \right] \notag \\
&= (A^{m-1})_{(i,1)} K_{n}(u,x) K_{n}(v^{-1},x^{-1}) \frac{1}{1 - \frac{x}{u}} \frac{1}{1 - \frac{v}{x}} \notag \\
&\quad \cdot \left[ \frac{u}{x^{2\cdot2^{n}}} + \frac{1}{x^{3\cdot2^{n}}} \left( u + \frac{u}{v} \right) + \frac{1}{x^{4\cdot2^{n}}} \left( 1 + 2\frac{u}{v} \right) \right. \notag \\
&\quad \left. + \frac{1}{x^{5\cdot2^{n}}} \left( v^{-1} + \frac{u}{v} \right) + \frac{1}{x^{6\cdot2^{n}} v} \right]. \label{eq:A-i,1}
\end{align}
By Lemma \ref{P-x,u}, we have
\begin{align}
&\CT_{u,v,x} (A^{m-1})_{(i,1)} K_{n}(u,x) K_{n}(v^{-1},x^{-1}) \frac{1}{1 - \frac{x}{u}} \frac{1}{1 - \frac{v}{x}} \cdot \frac{1}{x^{3\cdot2^{n}}} \left( u + \frac{u}{v} \right) \notag \\
&= \CT_{u,v} (A^{m-1})_{(i,1)} K_{n}(u,u) K_{n}(v^{-1},u^{-1}) \frac{1}{1 - \frac{v}{u}} \frac{1}{u^{3\cdot2^{n}}} \left( u + \frac{u}{v} \right). \label{eq:A-x}
\end{align}
In the polynomial expression
\[
(A^{m-1})_{(i,1)} K_{n}(u,u) K_{n}(v^{-1},u^{-1}) \frac{1}{1 - \frac{v}{u}} \frac{1}{u^{3\cdot2^{n}}} \left( u + \frac{u}{v} \right),
\]
the maximum power of $u$ is bounded above by $m - (2^{n} + 2 - n)$. This expression vanishes under the constant term operator since $m < 2^{n} + 2 - n$, implying (\ref{eq:A-x}) equals $0$. Similarly, the remaining terms in (\ref{eq:A-i,1}) vanish under $\CT_{u,v,x}$. Thus:
\begin{align}
\CT_{u,v,x} (A^{m-1})_{(i,1)} \varphi_{n+1}^{(-2)}
&= \CT_{u,v,x} (A^{m-1})_{(i,1)} K_{n}(u,x) K_{n}(v^{-1},x^{-1}) \frac{1}{1 - \frac{x}{u}} \frac{1}{1 - \frac{v}{x}} \frac{u}{x^{2\cdot2^{n}}} \notag \\
&= \CT_{u,v,x} (A^{m-1})_{(i,1)} \left( A_{(1,1)} \varphi_{n}^{(-2)} + A_{(1,2)} \varphi_{n}^{(-1)} + \cdots + A_{(1,5)} \varphi_{n}^{(2)} \right) \notag \\
&= \CT_{u,v,x} (A^{m-1})_{(i,1)} (A \Phi_{n})_{1}. \notag
\end{align}
Similarly for the other terms in (\ref{eq:A-i-row}), we obtain:
\begin{align*}
\CT_{u,v,x} (A^{m-1} \Phi_{n+1})_{i} = \CT_{u,v,x} \sum_{j=1}^{5} (A^{m-1})_{(i,j)} (A \Phi_{n})_{j}.
\end{align*}
Hence:
\[
\CT_{u,v,x} A^{m-1} \Phi_{n+1} = \CT_{u,v,x} A^{m} \Phi_{n}.
\]
Iterating this relation yields:
\[
\CT_{u,v,x} A^{m} \Phi_{n} = \CT_{u,v,x} A^{m-1} \Phi_{n+1} = \cdots = \CT_{u,v,x} \Phi_{m+n}. \qedhere
\]
\end{proof}

To compute $\omega(n)$, it is convenient to use the five states $\Psi_{n}:=(\psi_{n}^{(-2)},\psi_{n}^{(-1)},...,\psi_{n}^{(2)})^{t}$,
where $\psi_{n}^{(i)}:=x^{i\cdot2^{n}}G_{n}(x)G_{n}(x^{-1})$. Then we have

\begin{thm}\label{Recurrence-Matrix}
For all $m, n\in \mathbb{Z}_{+}$ with $m<2^{n}+2-n$, we have
$$\CT\limits_{x}\Psi_{m+n}=\CT\limits_{x}B^{m}\Psi_{n},$$
where the transfer matrix $B$ is give by
$$B= A|_{u=x,v=x} = \left( \begin {array}{ccccc} x&0&0&0&0\\ \noalign{\medskip}3&x+1&x&0&0
\\ \noalign{\medskip}{x}^{-1}&{x}^{-1}+1&3&x+1&x\\ \noalign{\medskip}0
&0&{x}^{-1}&{x}^{-1}+1&3\\ \noalign{\medskip}0&0&0&0&{x}^{-1}
\end {array} \right).$$
\end{thm}
\begin{proof}
Proposition \ref{Recurrence-Matrix-uv} tells us that $$\CT\limits_{u,v,x}\Phi_{m+n}=\CT\limits_{u,v,x}A^{m}\Phi_{n} \Rightarrow \CT\limits_{x}\left(\CT\limits_{u,v}\Phi_{m+n}\right)=\CT\limits_{x}\left(\CT\limits_{u,v}A^{m}\Phi_{n}\right).$$ It is clear that $\CT\limits_{u,v}\Phi_{m+n}=\Psi_{m+n}$ and $\CT\limits_{u,v}A^{m}\Phi_{n}=B^{m}\Psi_{n}$, so we get this corollary.
\end{proof}
By the conditions in Theorem \ref{Recurrence-Matrix}, we had better choose $m=9986, n=14$ so that $m+n=10000$ and $m<2^{14}+2-14$. In other words,
we choose $\Psi(14)$ as the initial term to obtain
$$
\omega(10000)=\CT\limits_{x}(\Psi_{10000})_{3}=\CT\limits_{x}(B^{9986}\Psi_{14})_{3}.
$$

\section{Algebraic Generating Functions}\label{4}
We first express $w(10^4)$ as the constant term of a two-variable rational generating function in $C[x,x^{-1}][[t]]$, the ring of power series in $t$ with coefficients Laurent polynomials in $x$. By taking constant term in $x$ first, we express
$w(10^4)$ as the constant term of an algebraic generating function in $C((t))$.
Finally we deduce P-recursions to finish the computation of $w(10^4)$.

\subsection{Reduction to the constant terms of rational functions}
We employ generating functions to circumvent explicit matrix multiplication. The generating function for the sequence is given by
the third component of the following:
\begin{align*}
\sum_{n\geq 0} (B^{n} \Psi_{14}) t^{n}
&= \sum_{n\geq 0} B^{n} \Psi_{14}  t^{n}  \\
&= (I - Bt)^{-1} \Psi_{14}  \\
&= G_{14}(x) G_{14}(x^{-1}) (I - Bt)^{-1}
\begin{pmatrix}
x^{-2 \cdot 2^{14}} \\
x^{-2^{14}} \\
1 \\
x^{2^{14}} \\
x^{2 \cdot 2^{14}}
\end{pmatrix} .
\end{align*}

Direct computation by Maple gives
\begin{align*}
\left( (I - Bt)^{-1}
\begin{pmatrix}
x^{-2 \cdot 2^{14}} \\
x^{-2^{14}} \\
1 \\
x^{2^{14}} \\
x^{2 \cdot 2^{14}}
\end{pmatrix} \right)_{3}
= & L(x,t) H(x,t),
\end{align*}
where
\[
L(x,t)= \sum_{i=0}^4 L_i(x)t^i =1+ \left( \sum_{k=-32769}^{32769} *\cdot x^{k} \right) t + \cdots + \left( \sum_{k=-32771}^{32771} *\cdot x^{k} \right) t^{4},
\]
and
\[
H(x,t) := -\frac{1}{(\frac{t}{x} - 1)(tx-1)\left((t^{3}-3t^{2}+t)x + (2t^{3}-6t^{2}+5t-1)+(t^{3}-3t^{2}+t)\frac{1}{x}\right)}
\]
is a power series in $t$ with coefficients Laurent polynomials in $x$.

Then we have: 
\[
\omega(10000) = \CT_{x,t} \, t^{-9986} \cdot H(x,t) L(x,t) G_{14}(x) G_{14}(x^{-1}).
\]

Expanding the product $L(x,t) G_{14}(x) G_{14}(x^{-1})$ yields:
\[
L(x,t) G_{14}(x) G_{14}(x^{-1}) = \sum_{k=0}^{4}\sum_{m=-65551}^{65551} c_{k.m}\cdot t^{k} x^{m},
\]
, which gives:
\[
\omega(10000) =  \sum_{k=0}^{4}\sum_{m=-65551}^{65551} c_{k,m} \cdot\CT_{x,t}t^{-9986+k}x^{m}H(x,t).
\]
Thus we are reduced to compute a sequence of subproblems, each of the form
$$\CT_{x,t}t^{k}x^{m}  H(x,t),   \qquad  -9986\leq k \leq -9982  \text{ and } -65551\leq m \leq 65551 .$$
The number of subproblems is approximately 589,967, which exceeds practical computational limits for subsequent calculations.

Using the property
$$H(x,t)=H(x^{-1},t) \Rightarrow  \CT_{x,t}x^{m} t^{k} H(x,t) = \CT_{x,t} x^{-m} t^{k} H(x,t), $$
we can remove positive powers in $x$ except that in $H(x,t)$.
Write
$$ H(x,t)=\sum_{n\geq0}h_{n}(x)t^{n}.$$
Observe that $h_{n}(x)$ satisfying  is a Laurent polynomial in $x$ with exponents ranging from $-n$ to $n$, and satisfies the symmetry $h_n(x) = h_n(x^{-1})$.

The additional property $\CT_x h_n(x) x^{-m} = 0$ for $m > n$ allows us to discard more terms. With the help of Maple, we arrive at
\begin{align}
\omega(10000) = \sum_{i=0}^{4} \sum_{j=-9986+i}^{0} a_{i,j} \CT\limits_{x,t} t^{-9986+i} x^j H(x,t). \label{eq:omega}
\end{align}
The number of subproblems is substantially decreased to at most 40000.

\subsection{The use of the quadratic generating series $X(t)$}
To extract the coefficients of $H(x,t)$, we perform partial fraction decomposition with respect to $x$. Precisely, let
\begin{align}
  X = X(t) &= \frac {1-2{t}^{3}+6
{t}^{2}-5t-\sqrt {12{t}^{4}-40{t}^{3}+33{t}^{2}-10t+1}+5t-1}{2t( 1-3t+{t}^{2}) } \label{eq:X}\\
&=t+2\,{t}^{2}+O \left( {t}^{3} \right) \notag
\end{align}
denote the unique quadratic power series solution to the denominator of $H(x,t)$. Then
\begin{align}
H(x,t) &= -\frac{x^2}{(x - t)(t^3 - 3t^2 + t)(x - X)(x - X^{-1})(tx - 1)} \notag \\
&= \frac{1}{t^2 - 3t + 1} \left( \frac{C_1}{x-t} + \frac{C_2}{1-tx} + \frac{C_3}{x-X} + \frac{C_4}{1-xX} \right), \label{eq:partial_fraction}
\end{align}
where $C_1$ and $C_3$ will be shown irrelevant for our purposes, and
\[
C_2 = \frac{X}{(1-t^{2})(X-t)(Xt-1)}, \quad C_4 = \frac{X^{2}}{t(X-t)(X^{2}-1)(Xt-1)}.
\]
Consequently, for $j \leq 0$, by Lemma \ref{P-x,u}, we have
\begin{align}
\CT_x H(x,t) x^j = U_j + V_j, \label{eq:CT_result}
\end{align}
with
\begin{align*}
U_{j} &= \frac{X^{2-j}}{t(t^{2}-3t+1)(X-t)(X^{2}-1)(Xt-1)}, \\
V_{j} &= -\frac{X t^{-j}}{(t^{2}-3t+1)(X-t)(Xt-1)(t^{2}-1)}.
\end{align*}
It follows that
\[
\CT_{x,t} t^{-9986+i} H(x,t) x^{j} = \CT_{t} \left( t^{-9986+i} U_{j} \right) + \CT_{t} \left( t^{-9986+i} V_{j} \right).
\]
Thus, the problem reduces to computing coefficients in the $t$-series expansions of $U_j$ and $V_j$.

The analysis of $U_j$ is complicated by the presence of $X^{-j}$. Since $U_0$ begins at $t^{-1}$, we write $U_j(t) = \sum_{k \geq -1} u_{j,k} t^k$. Expressing $U_j = X^{-j} N(t,X)/D(t,X)$, we compute
\[
\frac{d}{dt}U_{j}(t) = X^{-j}\frac{N_{1}(t,X)}{D_{1}(t,X)}, \quad \frac{d^{2}}{dt^{2}}U_{j}(t) = X^{-j} \frac{N_{2}(t,X)}{D_{2}(t,X)}.
\]
As $X$ is algebraic of degree $2$ by \eqref{eq:X}, we have the representations:
\begin{align*}
N(t,X)D_{1}(t,X)D_{2}(t,X)/X^{-j} &= P_{0}(t)X + Q_{0}(t), \\
N_{1}(t,X)D(t,X)D_{2}(t,X)/X^{-j} &= P_{1}(t)X + Q_{1}(t), \\
N_{2}(t,X)D(t,X)D_{1}(t,X)/X^{-j} &= P_{2}(t)X + Q_{2}(t).
\end{align*}
We then determine polynomials $f_0(t), f_1(t), f_2(t)$ satisfying
\[
f_0 P_{0} + f_1 P_1 + f_2 P_2 = 0, \quad f_0 Q_{0} + f_1 Q_1 + f_2 Q_2 = 0,
\]
yielding the differential equation
\[
f_0(t)U_{j}(t) + f_1(t)\frac{d}{dt}U_{j}(t) + f_2(t)\frac{d^{2}}{dt^{2}}U_{j}(t) = 0.
\]
This generates a P-recurrence for the coefficients $u_{j,k}$ at fixed $j$:
\begin{align*}
\left(2j^{2} - 2(k+1)^{2} \right)u_{j,k} &+ (-40j^{2} + 48k^{2} + 20k + 8)u_{j,k-1} \\
&+ \cdots + (72k^{2} - 1368k + 6480)u_{j,k-14} = 0.
\end{align*}
The leading coefficients of $X$ provide initial conditions:
\[
u_{j,-j-1} = \frac{1}{2}, \quad u_{j,-j-2} = 0, \quad \ldots, \quad u_{j,-j-14} = 0,
\]
together with the recursion, enabling efficient computation of $u_{j,-9986+i}$.

The $V_j$ case is more straightforward, since $\CT_{t} t^{-9986+i} V_{j} = \CT_{t} t^{-9986+i-j} V_{0}$. We expand
\[
V_0(t) = \sum_{k \geq 0} v_k t^k.
\]
A parallel derivation yields a P-recurrence for $v_k$ (details omitted), which similarly facilitates rapid computation.

Therefore,
\begin{align}
\CT_{x,t} t^{-9986+i} H(x,t) x^{j} &= \CT_{t} t^{-9986+i} (U_{j} + V_{j}) \notag \\
&= u_{j,-9986+i} + v_{-9986+i-j}. \label{eq:ij}
\end{align}

Combining \eqref{eq:omega} and \eqref{eq:ij}, we obtain $\omega(10000)$.

\section{Result}
Using the definition, Maple can compute the first 28 values of $\omega(n)$, starting at $ n=1$, relatively quickly. However, when it comes to larger values, it becomes very slow to complete the run. With the method mentioned in this paper, we can get these values quickly and they are the same as those values we get from the definition. So far, we have verified the correctness of this method in small cases.

Then we use Maple to implement those procedures in Section \ref{4} to get $\omega(10000)$, which costs us about 30 hours. Finally, we get the exact value of $\omega(10000)$, a certain 6591-digit number. See the Appendix. We also provide a Maple sheet of this work at \url{https://pan.baidu.com/s/1PA0_DcwZBgrRBni37MeLoA?pwd=DENU}.

\section{Concluding Remarks}
In this paper, we present a constant term approach to solve the challenging problem posed by Ekhad and Zeilberger regarding the computation of $\omega(10000)$.
Although our method yields this specific result, it does not produce a polynomial-time algorithm for computing $\omega(n)$ for arbitrary $n$.

The developed technique also applies to
similar modified problems, such as replacing $G_n(x)$ with
$$\prod_{i=0}^{n-1} \left(1 + x^{2^i + \epsilon_1} + x^{2^{i+1} + \epsilon_2}\right), $$
where $\epsilon_1, \epsilon_2$ are fixed integers. Proposition \ref{Recurrence-Matrix-uv} should be adjusted accordingly for such cases.

Furthermore, we may apply the constant term method to $u_{\alpha}$ as defined in Equation ~\eqref{eq:u-alpha}. For instance, to compute
$u_3(n) = \sum_{k} a(n,k)^3$, we observe that
$$u_3(n) = \CT_{x,y} F(x)F(y)F(x^{-1}y^{-1}).$$
The computation then proceeds analogously to the approach detailed in Section \ref{2}.

\noindent
{\small \textbf{Acknowledgements:}
The authors would like to thank Doron Zeilberger for his encouragement and valuable suggestions for improving the presentation.

\newpage
\section*{Appendix A: The exact value of $\omega(10000)$}

This is a 6591-digit number as below:

\seqsplit{675076678550698325901796479403600251143360381952601415535642205490109843721680975356704568094370023230114068379355174312057503573495958517311425020825180891240234025630019655597937920600915365471092417343506513602225462887927712738276088914336391550933547158940200788446297923610569534557796061504814214117593555014568157052382104644648070605187952034338123967238293088898229061800859158320141494952173485862431887426848644990177269679714070558927815758546404013701463762554967463419375028239658367580883536119898133224777038036032030692686111099636114179820820216063954391934925866570304805584864224486311695950010023768461803285862021942100020027736460660844967241539830842719469928778882320009600533149848374750041050670759042185564861253452049785498234971594653866307496116009901352119537411484121575772083460684319760159400126747926273038107748976702964967812838278704101126209958057529602807675732411032853230393932763367543423676794466492355094524238227291198564748382823276795465663668655901909925150834550595716786764781161217686876260375274173311813977093964917156064335134070356260593894148221480082348602315910691532559679976729784274166789626753539982288057368912926626449073546253157812368679914757117645630448437547448648349088766172141870848008921672915873317910861931669980736244456505950185488582609297995825439672308412561502256866407227461175874081136259659229294203040201952271966575549767718738399614442056325123203813922978610891422812910112356068717301473575503321920225724318798982069280119866204224148442621835233371466098246905545834271098359128633533487655514141248696565997476297075564605854876692099519648970676192612558672577909492371841499638873833334653309301343836927509385344463403207281387566604192368756093546762330224549068265512081789176884403321199829196833881988177033210566657928358770212118946012983199693263344477879079933171343128901499745402110803569837542365995085365516738184065509037740835725279484056733554483990336819090904543818245615709871759337466399802852067189988470637678018221756222211273526084703577299213957190912758468333320637898254539424577998414930114725215296342305265706764253148922848448494399486749042746624917615910625154524619753111211892747795567323710763495644280964102234890952310717190737072528456534466109368776429711607043473386913842244440310567552256122953326979865123089406071382259805961237434576498869645742184388786526935007619675925328207425173655698123025348616360223623659514680035483831453725981343930727681707070376029923464826365949058799183507104680546297757637540339590663642428010658157328360182530262899565976102754429053704406365023941972988273576870001165704358750402776786727136317366991701218846783281033233754477740790045081771210454495146329570046870240946480492785156366314199054966390708153172704247170238642883412179677571613009798131374206436624295635363141492826856773475324209658545970593442663203301437384235068792776995012378316336811493771573325238290034067520267420778627103379064838391679373124920078515644013444741713216338739605829363014954045286137914890643623952748609410043051635107645273144947081490328667018845135691352135627586383337406272670992471731962062968761750543764770258601110025024849349913096967808945311703796900943686398094837328551838801888211935741410035125488958165889109654853114984993308424179445008603086390071402101332470541582883615696821604502609217484900197693822373013454056308068867062953962401222774087850544755566507659795246984748276737390211612054304173125725229426322212146104271216996470286220334369157384908963885607801736482307313762652351240659233851264617145539110477598791472462714755026538041683179692801610945273950179185023018684314726135147487079641022813650636886220238956937544761390853053335491682306118890466481007890571307715473818844766888007691312517199031309081092400286626982906471024252750248431569148310862307081275199109821057683842183447249863143752639042983380156128092459962300304087537292029757137022199431244880659790067312580996266780527691748962928583693041590949725312977941063431073552784390251876161711125983092126443988293035567356525009657027429130167225510518493836482724597255004603952681538367375873668006273260302301037316131502267758455883790734958762604169006014824686894975973946032261317135044794309395387277562172031450783751314766996163222983547607964514719295851194456320005634542417209716022522785508794621737092318655045778173062204283110164002294813646566136948698838454335263685854934845476208041385988682489080098722472905595832197687137256523631463264113865513806043836047766267349266537359866244538506293641763544370124501213490804886618092499714950472211093349311804389049510501420937277375226929456770460380409557782041399178048782292046674673929508683147635536648313847569514170475522853058744968330115395454163515254717610183232531303321453385476703943316324968647183301082669490869258855129521532016994644174722753531013627225616937883229657097797238753113720870114908171497407336127784072116886471018178676325875208967558496472902910028937420153348938376460443716874352048677168225104161783704266685443150548686498421317852568779957727089793134759069309271894878664670728655067109316444402290945155384923832324317473309212595097734463493353792507277034928314768765433264511672133660214099301666662143121144483513319757949907300006530490357344498696911072570522239666372251834380413328963142953232896470693411167036288338585748687588517825136387293186755038048574005770457089127338286952707381557005825065792003758299220046424535148478918190960484076780881619970910446412521143973960219350133888199383735986968621507482926305527900685341581187658697297496425770914984230066046033622807669959409445279991294285648944528283378185724141297925931610099542028290274405224849781486965093251569516306389186580273702893763928908132954163864063738818033693579089028221077192755328763154865319336822288393969364604078377620635522721687888382193015338993764788543396203523007894718422375232631546106765581628500480754696728871383592423925135282746150000048033710440046758131697525201709293528311299947356760796086903176541326470690548067322542218922791629447885201297382748051872267014735771790202189364025545357575543333462762789433189932426311440990819097589922092007640217921803678861646347574348194165292156539517353067775951221045361172974627735199844909612963801398418386940998647247494542612226969688513408651641238357908891036980926053139058861325159810428770825121473339183130208808178861385826452682717215309961690281183200644238063334050625302764042460940707134568411256428692446447425131.}

\end{document}